\documentclass[12pt,a4paper,reqno]{amsart}
\usepackage{fancyhdr}
\usepackage{appendix}
\usepackage{amssymb,amscd,amsxtra,calc}
\usepackage{mathrsfs}
\usepackage{amsmath}
\usepackage{multirow}
\usepackage{verbatim}
\usepackage{mathtools}
\usepackage[all]{xy}
\usepackage[colorlinks,linkcolor=red,anchorcolor=blue,citecolor=blue]{hyperref}
\usepackage{tikz}
\usepackage{tikz-cd}
\theoremstyle{plain}
    \newtheorem{thm}{Theorem}[section]

    \newtheorem{lemma}[thm]{Lemma}
    \newtheorem{proposition}[thm]{Proposition}
    
    \newtheorem{theorem}[thm]{Theorem}

\theoremstyle{definition}
    \newtheorem{definition}[thm]{Definition}

    \newtheorem*{notation*}{Notation and Terminology}
    \newtheorem{remark}[thm]{Remark}

\theoremstyle{remark}

\makeatletter

\newcommand{\Rmnum}[1]{\expandafter\@slowromancap\romannumeral #1@}
\makeatother
\numberwithin{equation}{section}

\newcommand{\C}{\mathbb C}
\newcommand{\R}{\mathbb R}
\newcommand{\PP}{\mathbb P}
\newcommand{\Span}{\operatorname{span}}
\newcommand{\ord}{\operatorname{ord}}
\newcommand{\AS}{\mathrm{AS}}
\newcommand{\KS}{\mathrm{KS}}

\title[Proof of the planar Khavinson--Shapiro conjecture]
      {Proof of the planar Khavinson--Shapiro conjecture}
\author{Feng Shao}
\author{Weicheng Zhan}

\address{
\textsc{Feng Shao, School of Mathematical Sciences, Xiamen University, Xiamen 361005, P. R. China}
}
\email{shaofeng@amss.ac.cn, shaofeng@xmu.edu.cn}

\address{
\textsc{Weicheng Zhan, School of Mathematical Sciences, Xiamen University, Xiamen 361005, P. R. China}
}
\email{zhanweicheng@amss.ac.cn}

\begin{document}
\begin{abstract}
Let $\Omega\subset\R^2$ be a bounded domain whose boundary is a finite union of pairwise disjoint Jordan curves.
We prove the planar Khavinson--Shapiro conjecture in this setting: if every polynomial on $\R^2$ agrees on $\partial\Omega$ with a harmonic polynomial, then $\partial\Omega$ is an ellipse and $\Omega$ is its bounded interior.
\end{abstract}

\maketitle
\tableofcontents
\section{Introduction}

Ellipsoids have an exceptional algebraic property for the Dirichlet problem: every polynomial on their boundary agrees there with a harmonic polynomial.
This follows from Fischer's decomposition; see \cite{Fischer1918,Shapiro1989,Axler2004,Khavinson2014}.
In 1992, Khavinson and Shapiro conjectured that no other bounded domains have this property~\cite[p.~460]{Khavinson1992}.
The conjecture asks, in a particularly concrete form, how much of the geometry of a domain is encoded by exact polynomial solvability.

The hypothesis concerns polynomial data of arbitrary degree, whereas the conclusion is a single quadratic equation for the boundary.
In the plane, complex coordinates connect the two: after replacing $\bar z$ by an independent variable $w$, harmonic polynomials become sums $A(z)+B(w)$.
We use this separated form to prove the planar conjecture in the formulation of~\cite[Conjecture~1.1]{Tikaradze2021}.

Our main result is the following.

\begin{theorem}[{Planar Khavinson–Shapiro conjecture}]\label{thm1-1}
Let $\Omega\subset\R^2$ be a bounded domain whose boundary is a finite union of pairwise disjoint Jordan curves.
Suppose that, for every $f\in\R[x,y]$, there exists a harmonic polynomial $u_f\in\R[x,y]$ such that
\[
u_f=f\qquad\text{on }\partial\Omega.
\]
Then $\partial\Omega$ is an ellipse and $\Omega$ is its bounded interior.
\end{theorem}

Here and throughout, the term \emph{ellipse} includes the circular case.
In particular, Theorem~\ref{thm1-1} rules out holes as well as a nonelliptic outer boundary, although no smoothness is assumed.
Its proof rests on the following algebraic result about additive separation in the coordinate rings of plane curves.

\begin{theorem}\label{thm1-2}
Let
\[
F(z,w)=P(z)+Q(w)+czw\in\C[z,w], \qquad c\ne0, \qquad \deg P=\deg Q=d\ge3.
\]
Let $q$ be a nonconstant square-free divisor of $F$.
Write
\[
\mathscr A_q=\C[z,w]/(q),
\]
and let $\mathscr A_q^z$ and $\mathscr A_q^w$ be the images of $\C[z]$ and $\C[w]$ in $\mathscr A_q$.
Assume that every class modulo $q$ has a representative $A(z)+B(w)$, or equivalently that
\begin{equation}\label{1-1}
\mathscr A_q=\mathscr A_q^z+\mathscr A_q^w
\end{equation}
as complex vector spaces.
Then
\[
\deg q\le2.
\]
No irreducibility assumption is imposed on $q$.
\end{theorem}

We call the polynomial extension hypothesis in Theorem~\ref{thm1-1} property $(\KS)$.
For the balanced curves arising from $(\KS)$, Theorem~\ref{thm1-2} says that additive separation can hold only on divisors of degree at most two, even when the divisor is reducible.
This is the degree reduction needed to recover a quadratic boundary.

\subsection*{Earlier work and the present contribution}

Tikaradze's planar theorem assumes a degree-two algebraic-harmonic relation, formulated as the rational-function hypothesis in \cite[Theorem~1.1]{Tikaradze2021}.
Property $(\KS)$ applied to $|z|^2$, however, yields only
\[
zw-H(z)-H^*(w)=0
\]
on the complexified boundary, with no control on $\deg H$.
Thus Theorem~\ref{thm1-2} supplies the missing high-degree step: it establishes the degree conclusion predicted by \cite[Conjecture~1.2]{Tikaradze2021} for square-free,
possibly reducible divisors of $P(z)+Q(w)+czw$ with $\deg P=\deg Q$ and $c\ne0$.
The other alternative in that conjecture, linearity in one coordinate, is excluded here by the balanced bidegree as soon as the total degree exceeds two.

The additive-separation hypothesis is Tikaradze's KS-polynomial property~\cite[Definition~3.1]{Tikaradze2021}.
The subspace-sum argument used at the end also appears in~\cite[Proposition~2.1]{Tikaradze2021}.
Before that comparison can be applied, we must analyze the ratio group at infinity, construct a common parameter $t=U(z)=V(w)$ in the coordinate ring, and compute the rank of the coordinate ring in the reducible case.

Related algorithmic work studies separated elements of polynomial ideals and separated representatives on algebraic curves \cite{Buchacher2020,Buchacher2024a,Buchacher2024b,Buchacher2024c}.
Those papers address specific ideals or functions, rather than separation of every class, and do not give a uniform bound $\deg q\le2$.
Render's multi-factor theorem~\cite{Render2017} assumes at least three nonconstant factors, two of which have a common zero; the active divisor used below need not have this form.

The broader algebraic background includes the Fischer-operator approach \cite{Ebenfelt2005,Chamberland2001,Render2008,Render2016},
work on analytic continuation and polynomial decomposition \cite{Lundberg2009,Lundberg2011}, and the connection with finite-term recurrences for planar orthogonal polynomials \cite{Putinar2007,Khavinson2010}.

\subsection*{Outline of the proof}

Property $(\KS)$ is used at two separate points.
The datum $|z|^2$ first places the boundary on an algebraic curve.
The remaining polynomial data then impose additive separation on the components that meet the boundary infinitely often.

Write
\[
z=x+iy,\qquad w=x-iy.
\]
The harmonic extension of $|z|^2$ has the form $H(z)+H^*(\bar z)$, so
\[
\Psi(z,w)=zw-H(z)-H^*(w)
\]
vanishes on the complexified boundary.
Let $q$ be the square-free product of the irreducible factors of $\Psi$ that meet the boundary infinitely often.
Applying B\'ezout's theorem one factor at a time shows that every class in $\C[z,w]/(q)$ has a representative $A(z)+B(w)$, as required in Theorem~\ref{thm1-2}.
We keep the full square-free product because different parts of the boundary may lie on different factors.

The proof of Theorem~\ref{thm1-2} takes place at infinity.
The leading form of $q$ is a product of distinct lines $w=sz$, and the ratios of their slopes generate a subgroup of the $d$-th roots of unity.
If this subgroup is the whole group, a descent in degree ends with a cubic class that cannot be separated on three distinct branches.
If it is proper, Laurent expansions produce polynomials $U$ and $V$ for which B\'ezout's theorem gives the exact relation
\[
q\mid U(z)-V(w).
\]
After localization at the common value $t=U(z)=V(w)\in \C[z,w]/(q)$, the two resulting dimension counts are incompatible.
Both arguments take place in the full reduced ring and therefore allow $q$ to be reducible.

The algebraic theorem gives $\deg q\le2$.
The last step is real and topological: the active curve is defined over $\R$, its real zero set is an ellipse, and the Jordan curve theorem identifies $\Omega$ with the bounded component of its complement.

Sections~\ref{sec:preliminaries} and~\ref{sec:active} carry out the passage from the boundary problem to the reduced coordinate ring.
Sections~\ref{sec:infinity}--\ref{sec:localization} prove the degree bound, and Section~\ref{sec:completion} returns to the real plane.

\section{Algebraic preliminaries}\label{sec:preliminaries}

\subsection{Complex coordinates and harmonic polynomials}

The complex-linear change of variables
\[
z=x+iy,\qquad w=x-iy
\]
identifies $\C[x,y]$ with $\C[z,w]$.
Under the real embedding
\begin{equation}\label{2-1}
\iota:\C\longrightarrow\C^2, \qquad \iota(z)=(z,\bar z),
\end{equation}
the variable $w$ becomes $\bar z$.
Moreover,
\[
\Delta=\partial_x^2+\partial_y^2 =4\partial_z\partial_w.
\]

\begin{lemma}\label{lem2-1}
A polynomial $h\in\C[x,y]$ is harmonic if and only if there exist $A,B\in\C[\zeta]$ such that
\[
h(x,y)=A(z)+B(\bar z).
\]
Moreover, if $h$ is harmonic and is real-valued on $\R^2$, the decomposition may be chosen in the form
\[
h=H(z)+H^*(\bar z),
\]
where $f^*(\zeta):=\overline{f(\bar\zeta)}$ for $f\in\C[\zeta]$.
\end{lemma}

\begin{proof}
Regard $h$ as a polynomial in the independent variables $z,w$.
Since $\Delta h=4\partial_z\partial_w h$, harmonicity is equivalent to $h=A(z)+B(w)$; the converse is immediate.
If $h$ is real-valued on $\R^2$, then
\[
A(z)+B(\bar z)=\overline{A(z)+B(\bar z)} =A^*(\bar z)+B^*(z).
\]
The real slice $\{(z,\bar z):z\in\C\}$ is Zariski dense in $\C^2$.
Hence
\[
A-B^*=A^*-B=C
\]
for some $C\in\C$.
Applying $*$ gives $C\in\R$, and the choice $H=A-C/2$ yields the stated form.
\end{proof}

The following quotient-ring condition records the form of separation that we shall use.

\subsection{Additive separation in a quotient ring}

\begin{definition}\label{def2-1}
For a nonconstant polynomial $q\in\C[z,w]$, set $\mathscr A_q:=\C[z,w]/(q)$, and let $\pi_q:\C[z,w]\to\mathscr A_q$ be the quotient map.
We say that $q$ has the \emph{additive separation property}, abbreviated $(\AS)$, if every class in $\mathscr A_q$ has a representative $A(z)+B(w)$ with $A\in\C[z]$ and $B\in\C[w]$.
Equivalently, with $\mathscr A_q^z:=\pi_q(\C[z])$ and $\mathscr A_q^w:=\pi_q(\C[w])$, one has
\begin{equation}\label{2-2}
\mathscr A_q=\mathscr A_q^z+\mathscr A_q^w
\end{equation}
as complex vector spaces.
\end{definition}

This is Tikaradze's KS-polynomial property \cite[Definition~3.1]{Tikaradze2021}, expressed in quotient-ring language.

Property $(\AS)$ also has a Fischer-operator formulation.
The operator $\partial_z\partial_w$ is surjective on $\C[z,w]$, and
\[
\ker(\partial_z\partial_w)=\C[z]+\C[w].
\]
Hence $(\AS)$ is equivalent to the surjectivity of
\[
\C[z,w]\longrightarrow\C[z,w], \qquad R\longmapsto\partial_z\partial_w(qR).
\]
Under the inverse change of variables, $\partial_z\partial_w=\frac14\Delta$.
Thus property $(\AS)$ is Render's criterion~\cite{Render2016} written in complex coordinates.
We use \eqref{2-2} below because it keeps the two summands visible.

\subsection{Projective curves and orders at infinity}

To test the quotient identity, we use the branches at infinity of $V(q)$, where $V(q)$ is the set of zero locus of $q$.
We record the required notation and the form of B\'ezout's theorem used below.

If $R\in\C[z,w]$ has total degree $r$, its homogenization is
\[
R^h(Z,W,T)=T^rR(Z/T,W/T).
\]
If two projective plane curves of degrees $r$ and $s$ have no common irreducible component, B\'ezout's theorem says that the sum of their local intersection multiplicities in $\PP^2_{\C}$ is $rs$;
see \cite[Theorem~I.7.7 and Corollary~I.7.8]{Hartshorne1977}.

Near a point at infinity with $Z\ne0$, we work in the chart $Z=1$ and use
\[
t=T/Z=1/z, \qquad v=W/Z=w/z.
\]
When the curve is smooth and transverse to $T=0$, $t$ is a local parameter.
For a Laurent series $L(t)$, the notation $L(t)=O(t^k)$ means that every term has exponent at least $k$.
Thus, for a degree-$r$ polynomial $R$ evaluated on a branch $(z(t),w(t))$,
\[
R^h(1,w(t)/z(t),t)=t^rR(z(t),w(t)).
\]

\section{From boundary data to additive separation}\label{sec:active}

The curve obtained from the datum $|z|^2$ may have components that meet the boundary only finitely many times.
Moreover, a single Jordan component of the boundary need not lie on one irreducible factor.
We therefore take the reduced product of all factors met infinitely often by the boundary.
This is the complex-coordinate form of the Fischer-operator construction in \cite[Theorem~1]{Render2016}; retaining every such factor will be important in the reducible case.

\subsection{The polynomial supplied by \texorpdfstring{$|z|^2$}{the modulus squared}}

Apply the hypothesis of Theorem~\ref{thm1-1} to
\[
f(x,y)=x^2+y^2=|z|^2.
\]
Let $h\in\R[x,y]$ be a harmonic polynomial satisfying $h=|z|^2$ on $\partial\Omega$.
By Lemma~\ref{lem2-1},
\[
h=H(z)+H^*(\bar z)
\]
for some $H\in\C[\zeta]$.
Consequently,
\begin{equation}\label{3-1}
\Psi(z,w):=zw-H(z)-H^*(w)
\end{equation}
vanishes on $\iota(\partial\Omega)$.
It is nonzero because the coefficient of $zw$ is $1$.

Factor $\Psi$ in $\C[z,w]$.
An irreducible factor $p$ of $\Psi$ is called \emph{active} if $V(p)\cap\iota(\partial\Omega)$ is infinite.
Choose one representative from each associate class of active factors and set
\begin{equation}\label{3-2}
q:=\prod_{p\ \mathrm{active}}p.
\end{equation}

\begin{lemma}\label{lem3-1}
The polynomial $q$ is a nonconstant square-free divisor of $\Psi$, and
\begin{equation}\label{3-3}
\iota(\partial\Omega)\subset V(q).
\end{equation}
\end{lemma}

\begin{proof}
The product in \eqref{3-2} is square-free by construction and divides $\Psi$.
There must be at least one active factor.
Indeed, if every irreducible factor met $\iota(\partial\Omega)$ in only finitely many points, then the zero set of $\Psi$ would meet the embedded boundary in a finite set,
whereas $\Psi$ vanishes on the infinite set $\iota(\partial\Omega)$.

Let
\[
E:=\iota(\partial\Omega)\setminus V(q).
\]
Every point of $E$ lies on an inactive factor of $\Psi$, so $E$ is finite.
Fix $\iota(z_0)\in E$.
The point $z_0$ lies on one of the Jordan components of $\partial\Omega$ and is therefore a limit of distinct points of that component.
Since $E$ is finite, we may choose such a sequence outside $E$.
Its image under $\iota$ lies in $V(q)$ and converges to $\iota(z_0)$.
The set $V(q)$ is closed, hence $\iota(z_0)\in V(q)$, a contradiction.
Therefore $E$ is empty and \eqref{3-3} follows.
\end{proof}

The construction of $q$ used only the datum $|z|^2$.
We now apply property $(\KS)$ to arbitrary polynomial data and obtain additive separation modulo $q$.

\subsection{Additive separation on the active curve}

Although Theorem~\ref{thm1-1} is stated for real polynomials, its hypothesis extends immediately to complex polynomial data.
Given $\Phi\in\C[z,w]$, set
\[
g(x,y)=\Phi(x+iy,x-iy).
\]
Write $g=g_1+ig_2$ with $g_1,g_2\in\R[x,y]$.
Apply the hypothesis to $g_1$ and $g_2$ and combine the resulting real harmonic polynomials.
By Lemma~\ref{lem2-1}, there are $A,B\in\C[\zeta]$ such that
\begin{equation}\label{3-4}
\Phi(z,\bar z)=A(z)+B(\bar z) \qquad (z\in\partial\Omega).
\end{equation}

\begin{proposition}\label{prop3-1}
The active polynomial $q$ satisfies $(\AS)$:
\[
\mathscr A_q=\mathscr A_q^z+\mathscr A_q^w.
\]
\end{proposition}

\begin{proof}
Fix $\Phi\in\C[z,w]$ and choose $A,B$ as in \eqref{3-4}.
Put
\[
G(z,w)=\Phi(z,w)-A(z)-B(w).
\]
If $G\equiv0$, then $q\mid G$ is immediate.
Assume henceforth that $G\ne0$.
Since $G$ vanishes on the infinite set $\iota(\partial\Omega)$, it is necessarily nonconstant.
For every active irreducible factor $p$, the set $V(p)\cap\iota(\partial\Omega)$ is infinite and is contained in $V(G)$.
If $p$ did not divide $G$, B\'ezout's theorem would imply that $V(p)\cap V(G)$ is finite~\cite[Corollary~I.7.8]{Hartshorne1977}.
Hence $p\mid G$.
Distinct active factors are pairwise coprime, so their product divides $G$:
\[
q\mid G=\Phi-A-B.
\]
Thus the class of $\Phi$ modulo $q$ has a representative $A(z)+B(w)$ and $(\AS)$ follows.
\end{proof}

We may now leave the boundary problem aside.
If $H=0$ or $\deg H\le2$, then $\deg q\le2$ because $q\mid\Psi$.
Otherwise, let $d=\deg H\ge3$.
Since $\deg H^*=d$, the polynomial \eqref{3-1} has the form required in Theorem~\ref{thm1-2}, with
\[
P=-H,\qquad Q=-H^*,\qquad c=1.
\]
Sections~\ref{sec:infinity}--\ref{sec:localization} prove that theorem.

\section{Branches at infinity and ratio groups}\label{sec:infinity}

At infinity, the leading form of $F$ is a binomial, so the leading form of each divisor is a product of distinct noncoordinate lines.
The slopes of these lines provide the test points for separated principal parts.

Throughout Sections~\ref{sec:infinity}--\ref{sec:localization}, let
\begin{equation}\label{4-1}
F(z,w)=P(z)+Q(w)+czw, \qquad c\ne0, \qquad \deg P=\deg Q=d\ge3,
\end{equation}
and let $q$ be a nonconstant square-free divisor of $F$.
Write
\[
m=\deg q.
\]
Let $a$ and $b$ be the leading coefficients of $P$ and $Q$, respectively.
The degree-$d$ homogeneous part of $F$ is
\begin{equation}\label{4-2}
F_d(z,w)=az^d+bw^d.
\end{equation}

\begin{lemma}\label{lem4-1}
There is a set $S\subset\C^*$ of cardinality $m$ and a constant $\rho\ne0$ such that
\begin{equation}\label{4-3}
q_m(z,w)=\rho\prod_{s\in S}(w-sz),
\end{equation}
where
\begin{equation}\label{4-4}
s^d=-a/b \qquad(s\in S).
\end{equation}
In particular:
\begin{enumerate}
\item the projective closure of $V(q)$ has exactly $m$ distinct smooth points at infinity, one for each $s\in S$;
\item every such branch is transverse to the line at infinity and can be written
\begin{equation}\label{4-5}
z=t^{-1},\qquad w=st^{-1}+O(1);
\end{equation}
\item $q$ has degree $m$ in each variable, with nonzero constant leading coefficient when viewed as a polynomial in either $z$ or $w$.
\end{enumerate}
\end{lemma}

\begin{proof}
Write $F=qr$.
Then
$\deg F=\deg q+\deg r$
and comparison of the highest homogeneous parts gives
\[
F_d=q_mr_{d-m}.
\]
The polynomial $az^d+bw^d$ is a product of $d$ distinct noncoordinate linear forms: its slopes are the simple nonzero roots of $\xi^d+a/b$.
Hence $q_m$ is the product of $m$ of these factors, which gives \eqref{4-3} and \eqref{4-4}.

In the chart $Z=1$, put $v=W/Z$ and $t=T/Z$.
The homogenization of $q$ has the form
\[
q^h(1,v,t)=\rho\prod_{s\in S}(v-s)+tR(v,t).
\]
At $(v,t)=(s,0)$, the derivative with respect to $v$ is nonzero.
The implicit function theorem gives a unique smooth branch $v=v_s(t)$ with $v_s(0)=s$, and $t$ is a local parameter on this branch.
This yields \eqref{4-5} in affine coordinates.
The coefficients of $w^m$ and $z^m$ in \eqref{4-3} are nonzero, while no lower homogeneous term can contain either monomial.
Part~(3) follows.
\end{proof}

Choose one $s_0\in S$.
Equation \eqref{4-4} says that
\[
S\subset s_0\mu_d,
\]
where $\mu_d$ is the cyclic group of $d$-th roots of unity.
Define the \emph{ratio group}
\begin{equation}\label{4-6}
G(S):=\langle s/s':s,s'\in S\rangle\le\mu_d.
\end{equation}
There is a unique positive divisor $e\mid d$ such that $G(S)=\mu_e$.
We shall treat the two cases
\[
e=d\qquad\text{and}\qquad e<d.
\]

\begin{lemma}\label{lem4-2}
If $e<d$, then $S$ is contained in one coset of $\mu_e$ and
\begin{equation}\label{4-7}
m\le e\le d/2.
\end{equation}
In particular, if $m\ge3$, then
\begin{equation}\label{4-8}
d-e-2\ge e-2\ge1.
\end{equation}
\end{lemma}

\begin{proof}
For $s\in S$, the quotient $s/s_0$ belongs to $G(S)=\mu_e$, so $S\subset s_0\mu_e$ and $m\le e$.
A proper subgroup of the cyclic group $\mu_d$ has order at most $d/2$.
If $m\ge3$, then $e\ge3$, and \eqref{4-8} follows from $d\ge2e$.
\end{proof}

The value of $e$ determines how the proof proceeds.
If $e=d$, cancellation among separated leading terms is governed by the full group $\mu_d$.
If $e<d$, the slopes lie in a single coset of the smaller group $\mu_e$, which leads instead to a common $e$-th power.
In both cases the comparison is made on the branches at infinity.

For later use, we record how a separated polynomial looks on these branches.
If
\[
A(z)=\sum_{k=0}^N\alpha_kz^k, \qquad B(w)=\sum_{k=0}^N\beta_kw^k,
\]
then the coefficient of $t^{-N}$ in $A(t^{-1})+B(st^{-1}+O(1))$ is
\begin{equation}\label{4-9}
\alpha_N+\beta_Ns^N.
\end{equation}
We call the vector $(\alpha_N+\beta_Ns^N)_{s\in S}\in\C^m$ the \emph{principal vector} of the degree-$N$ separated part.
More generally, if $w_s(t)=st^{-1}+O(1)$ denotes the branch of slope $s$, then
\begin{equation}\label{4-10}
\alpha_r=\beta_r=0\ (r>k) \quad\Longrightarrow\quad [t^{-k}]\bigl(A(t^{-1})+B(w_s(t))\bigr) =\alpha_k+\beta_ks^k,
\end{equation}
where $[t^{-k}]$ denotes coefficient extraction.
Terms of degree below $k$ have pole order less than $k$, and the displayed hypothesis excludes contributions from higher degrees.
We shall use \eqref{4-10} only after those higher terms have been removed.

\section{The case of full ratio group}\label{sec:full}

Assume that $G(S)=\mu_d$.
If the leading terms of a separated polynomial cancel on every branch, their degree must be divisible by $d$.
The relation $F=0$ then allows us to remove those terms.
Applying this reduction once either gives a contradiction at the next mixed degree or leaves a cubic comparison, for which three distinct slopes already suffice.

\begin{lemma}\label{lem5-1}
Let $S\subset\C^*$ contain at least three distinct points.
Then there exists $j\in\{1,2\}$ such that
\begin{equation}\label{5-1}
(s^j)_{s\in S}\notin \Span\bigl\{(1)_{s\in S},(s^3)_{s\in S}\bigr\}\subset\C^m.
\end{equation}
\end{lemma}

\begin{proof}
Put $v_i=(s^i)_{s\in S}$ for $0\le i\le3$.
The vectors $v_0,v_1,v_2$ are linearly independent: a relation $av_0+bv_1+cv_2=0$ would give a quadratic polynomial $a+bs+cs^2$ vanishing at three distinct points.
If both $v_1$ and $v_2$ belonged to $\Span\{v_0,v_3\}$, then the three-dimensional space $\Span\{v_0,v_1,v_2\}$ would be contained in a space of dimension at most two, a contradiction.
\end{proof}

\begin{lemma}\label{lem5-2}
Assume $m=|S|\ge3$ and $G(S)=\mu_d$.
Then $q$ does not satisfy $(\AS)$.
\end{lemma}

\begin{proof}
Choose $j\in\{1,2\}$ as in Lemma~\ref{lem5-1} and set
\begin{equation}\label{5-2}
\tau_j(z,w)=z^{3-j}w^j.
\end{equation}
Suppose, toward a contradiction, that $q$ satisfies $(\AS)$.
There are $A,B\in\C[\zeta]$ such that
\begin{equation}\label{5-3}
\tau_j\equiv A(z)+B(w)\pmod q.
\end{equation}
The polynomials $A$ and $B$ cannot both vanish.
Indeed, otherwise $q\mid\tau_j$.
Since $m=\deg q\ge3=\deg\tau_j$, this would force $m=3$ and $q_3\mid\tau_j$.
That is impossible because $q_3$ is a product of three distinct noncoordinate lines, whereas $\tau_j$ is a monomial.
Let $N$ be the maximum of the degrees of the nonzero polynomials among $A$ and $B$.
Suppose first that $N>3$.
Since the left-hand side of \eqref{5-3} has pole order three on each branch at infinity, the order-$N$ principal vector on the right must vanish:
\begin{equation}\label{5-4}
\alpha_N+\beta_Ns^N=0 \qquad(s\in S).
\end{equation}
Neither $\alpha_N$ nor $\beta_N$ can be zero; otherwise \eqref{5-4} would force both to be zero, contrary to the definition of $N$.
Hence $s^N$ is constant on $S$.
It follows that
\[
(s/s')^N=1 \qquad(s,s'\in S).
\]
Because these ratios generate $\mu_d$, every $d$-th root of unity has $N$-th power one.
Thus
\begin{equation}\label{5-5}
d\mid N.
\end{equation}
Write $N=d\ell$, with $\ell\ge1$.
To remove the degree-$N$ terms using the relation $F=0$, set
\begin{equation}\label{5-6}
J_\ell(z,w)=P(z)^\ell-(-Q(w))^\ell.
\end{equation}
Since $q\mid F$, we have $P+Q\equiv-czw\pmod q$, and therefore
\begin{align}
J_\ell&=\bigl(P+Q\bigr)\sum_{r=0}^{\ell-1}P^{\ell-1-r}(-Q)^r \notag\\
&\equiv-czw\sum_{r=0}^{\ell-1}P^{\ell-1-r}(-Q)^r\pmod q.\label{5-7}
\end{align}
The degree-$N$ terms of $J_\ell$ are
\[
a^\ell z^N-(-b)^\ell w^N.
\]
They have the same coefficient ratio as the degree-$N$ terms of $A+B$.
Indeed, each $s\in S$ satisfies $s^d=-a/b$, so
\[
a^\ell-(-b)^\ell s^{d\ell}=0,
\]
whereas \eqref{5-4} determines the same one-dimensional space of coefficient pairs.
More explicitly,
\[
\bigl\{(x_1,x_2)\in\C^2:x_1+x_2s^N=0\text{ for every }s\in S\bigr\} \cong \C.
\]
It follows that there is a unique $\gamma\ne0$ for which
\begin{equation}\label{5-8}
A_1=A-\gamma P^\ell, \qquad B_1=B+\gamma(-Q)^\ell
\end{equation}
have no terms of degree $N$.
In the quotient ring $\C[z,w]/(q)$,
\begin{equation}\label{5-9}
A_1(z)+B_1(w) \equiv\tau_j-\gamma J_\ell \equiv\tau_j+\gamma czw \sum_{r=0}^{\ell-1}P^{\ell-1-r}(-Q)^r \pmod q.
\end{equation}

We first rule out the case $\ell\ge2$.
The highest mixed term introduced by the correction has degree $N-d+2$, and its restriction to a branch of slope $s$ is linear in $s$.
Put
\begin{equation}\label{5-10}
N_1=N-d+2.
\end{equation}
The highest homogeneous part of the right-hand side of \eqref{5-9} is
\[
\gamma czw\sum_{r=0}^{\ell-1} (az^d)^{\ell-1-r}(-bw^d)^r.
\]
On the line $w=sz$, this homogeneous polynomial is $\gamma c\ell a^{\ell-1}s z^{N_1}$, which is nonzero.
The second term on the right-hand side of \eqref{5-9} therefore has degree $N_1$, and its coefficient of order $t^{-N_1}$ on the branch of slope $s$ is
\begin{equation}\label{5-11}
\gamma c\ell a^{\ell-1}s.
\end{equation}
Since $\ell\ge2$, we have $N_1\ge5$, so the cubic $\tau_j$ does not affect this coefficient.

We claim that $A_1+B_1$ has no term of degree strictly between $N_1$ and $N$.
If not, let $k$ be the largest such degree.
Since the right-hand side of \eqref{5-9} has degree $N_1<k$, the coefficient of $t^{-k}$ on the left-hand side must vanish on every branch.
By the maximality of $k$, no higher pure term contributes at this order, and hence
\[
\alpha_k'+\beta_k's^k=0 \qquad(s\in S)
\]
where $\alpha_k'$ and $\beta_k'$ are the coefficients of $A_1$ and $B_1$ of degree $k$ terms, respectively.
Neither coefficient can be zero, so $s^k$ is constant on $S$.
Hence $d\mid k$.
But
\[
N-d+2<k<N
\]
contains no multiple of $d$: the preceding multiple is $N-d$, already below $N_1$.
This contradiction proves the claim.

There are now no terms of $A_1+B_1$ of degree strictly greater than $N_1$.
Comparison of the coefficient of $t^{-N_1}$ in \eqref{5-9} therefore involves only the degree-$N_1$ terms.
Since
\[
N_1=N-d+2\equiv2\pmod d
\]
and $s^d=-a/b$ is independent of $s\in S$, the principal vectors of degree $N_1$ lie in
\[
\Span\{(1)_{s\in S},(s^2)_{s\in S}\}.
\]
Equation \eqref{5-11} would therefore imply
\[
(s)_{s\in S}\in\Span\{(1)_{s\in S},(s^2)_{s\in S}\}.
\]
Since a polynomial of degree at most two cannot vanish at three distinct points of $S$, the vectors $(1)$, $(s)$, and $(s^2)$ are linearly independent.
This rules out $\ell\ge2$.

It remains, under the assumption $N>3$, to consider $\ell=1$.
In this case $N=d>3$.
Subtracting the multiple of $J_1=P+Q$ in \eqref{5-8} removes the degree-$d$ terms.
Since $J_1\equiv-czw\pmod q$, the right-hand side of \eqref{5-9} is $\tau_j$ plus a polynomial of degree two.
Suppose that a pure term of degree greater than three remains in $A_1+B_1$.
If $k$ is its largest degree, then comparison of the coefficient of $t^{-k}$ on the branches gives
\[
\alpha_k'+\beta_k's^k=0 \qquad(s\in S).
\]
Both coefficients are nonzero by the choice of $k$, so $s^k$ is constant on $S$.
Since the ratios of the elements of $S$ generate $\mu_d$, this forces $d\mid k$, which is impossible for $3<k<d$.
Hence no such term remains.
Comparison at pole order three now gives
\[
(s^j)_{s\in S}\in\Span\{(1)_{s\in S},(s^3)_{s\in S}\},
\]
contrary to lemma \ref{lem5-1}.

Finally, suppose that $N\le3$.
Comparing the terms of pole order three in \eqref{5-3} gives
\[
(s^j)_{s\in S}\in\Span\{(1)_{s\in S},(s^3)_{s\in S}\}.
\]
This also covers $d=3$ and $N=3$: in that case $S$ consists of all three cubic slopes, so $(s^3)_{s\in S}$ is constant while $(s^j)_{s\in S}$ is not.

The class of $\tau_j$ therefore has no separated representative modulo $q$, contrary to $(\AS)$.
\end{proof}

\iffalse
\begin{remark}\label{rem5-1}
The proof does not use irreducibility of $q$.
All coefficient comparisons are made on the local branches at infinity, using the same global coefficients of $A$ and $B$.
\end{remark}
\fi

\section{The case of proper ratio group}\label{sec:proper}

Suppose now that the ratio group is a proper subgroup.
The cubic comparison used in the previous section is no longer sufficient, but the slopes lie in a single coset of $\mu_e$.
This makes the leading $e$-th powers of the corresponding Laurent roots agree and eventually yields a separated polynomial divisible by $q$.

Assume
\begin{equation}\label{6-1}
G(S)=\mu_e,\qquad e<d,
\end{equation}
and $m\ge3$.

\subsection{Laurent roots}

Choose $d$-th roots $\alpha^d=a$ and $\beta^d=-b$.
There are unique formal Laurent series of the form
\begin{equation}\label{6-2}
\begin{split}
X(z)&=\alpha z+\alpha_0+\alpha_{-1}z^{-1}+\cdots,\\
Y(w)&=\beta w+\beta_0+\beta_{-1}w^{-1}+\cdots
\end{split}
\end{equation}
satisfying
\begin{equation}\label{6-3}
X(z)^d=P(z), \qquad Y(w)^d=-Q(w).
\end{equation}
Factoring out the leading monomial and applying the formal binomial expansion gives these series and their uniqueness.

All Laurent expansions below are taken in $\C((t))$ after substituting $z=t^{-1}$ and $w=t^{-1}v_s(t)$ on a branch of slope $s$.
In particular, we have
\begin{equation}\label{6-4}
\ord_t z(t)=\ord_t w(t)=-1, \qquad tX(z(t)),\ tY(w(t))\in\C[[t]]^\times.
\end{equation}
Let $p$ be the irreducible factor carrying the branch, let $\xi=[1:s:0]$, and let $E\in\C[z,w]$ have degree $r$.
If $E^h$ does not vanish identically on $V(p^h)$, the standard local formula for intersection multiplicity~\cite[Chapter~I, Sections~5 and 7]{Hartshorne1977} yields
\begin{equation}\label{6-5}
\begin{split}
I_\xi(p^h,E^h)&=\ord_t E^h(1,v_s(t),t)\\
&=\ord_t\!\left(t^rE(t^{-1},t^{-1}v_s(t))\right).
\end{split}
\end{equation}
Thus homogenization adds $r$ to the $t$-adic order of the affine expression.

For a Laurent series in $z$ or $w$, write $[\,\cdot\,]_+$ for its polynomial part, including the constant term.
Then
\begin{equation}\label{6-6}
X(z)^e-[X(z)^e]_+=O(z^{-1}), \qquad Y(w)^e-[Y(w)^e]_+=O(w^{-1}).
\end{equation}

On a branch at infinity, the quotient of the two Laurent roots tends to a $d$-th root of unity.
The proper-ratio assumption identifies the $e$-th powers of these limits, while the term $czw$ controls the remaining error.

\begin{lemma}\label{lem6-1}
There is a constant $\lambda\in\C^*$ such that the degree-$e$ polynomials
\begin{equation}\label{6-7}
U(z)=[X(z)^e]_+, \qquad V(w)=\lambda[Y(w)^e]_+
\end{equation}
satisfy, on every branch of $V(q)$ at infinity,
\begin{equation}\label{6-8}
R(z,w):=U(z)-V(w)=O(t).
\end{equation}
\end{lemma}

\begin{proof}
Fix a branch of slope $s\in S$ and use the parameterization \eqref{4-5}.
Put
\[
r(t)=\frac{X(z(t))}{Y(w(t))}.
\]
By \eqref{6-4}, the numerator and denominator both have valuation $-1$, so $r(t)$ is a unit of $\C[[t]]$.
Equation \eqref{6-2} gives
\[
r(t)=\rho_s+O(t), \qquad \rho_s:=\frac{\alpha}{\beta s}.
\]
The choices of $\alpha$ and $\beta$, together with \eqref{4-4}, give $\rho_s^d=1$, while the definition of $\rho_s$ gives
\[
\frac{\rho_s}{\rho_{s'}}=\frac{s'}s.
\]
The ratios $\rho_s/\rho_{s'}=s'/s$ generate $\mu_e$.
Hence the values $\rho_s^e$ are independent of $s$; write their common value as $\lambda\ne0$:
\begin{equation}\label{6-9}
\rho_s^e=\lambda \qquad(s\in S).
\end{equation}

Because $q\mid F$, the equation $F=0$ holds on the branch.
Using \eqref{6-3}, we obtain
\begin{equation}\label{6-10}
r^d-1 =\frac{X^d-Y^d}{Y^d} =-\frac{czw}{Y^d} =O(t^{d-2}).
\end{equation}
Equivalently, $\ord_t(r^d-1)\ge d-2$.
The limiting root $\rho_s$ of $T^d-1$ is simple.
Factoring $r^d-1=(r-\rho_s)\Theta(r)$ with $\Theta(\rho_s)\ne0$ gives
\begin{equation}\label{6-11}
r-\rho_s=O(t^{d-2}),
\end{equation}
and hence, by \eqref{6-9},
\begin{equation}\label{6-12}
r^e-\lambda=O(t^{d-2}).
\end{equation}
Multiplying by $Y^e=O(t^{-e})$ yields
\begin{equation}\label{6-13}
X(z(t))^e-\lambda Y(w(t))^e =O(t^{d-e-2}).
\end{equation}
By Lemma~\ref{lem4-2}, the exponent in \eqref{6-13} is at least one.
The discarded tails in \eqref{6-6} are also $O(t)$, because both $z(t)^{-1}$ and $w(t)^{-1}$ have positive $t$-adic valuation.
Subtracting these tails proves \eqref{6-8}.

Since the leading terms of $U$ and $V$ are $\alpha^ez^e$ and $\lambda\beta^ew^e$, respectively, both polynomials have degree $e$.
\end{proof}

\subsection{From contact to divisibility}

For $R=U-V$, equation \eqref{6-8} gives one additional order of vanishing after degree-$e$ homogenization.
B\'ezout's theorem then shows that each irreducible factor of $q$ divides $R$.

\begin{lemma}\label{lem6-2}
Under \eqref{6-1}, if $m\ge3$, then the polynomials $U,V$ from Lemma~\ref{lem6-1} satisfy
\begin{equation}\label{6-14}
q(z,w)\mid U(z)-V(w).
\end{equation}
\end{lemma}

\begin{proof}
Let $p$ be an irreducible factor of $q$ and put $m_p=\deg p$.
The argument of Lemma~\ref{lem4-1}, applied to $p$, shows that its leading form is the product of the $m_p$ linear factors of $q_m$ belonging to $p$.
Its homogenization $p^h$ is irreducible: otherwise a homogeneous factorization, after setting $T=1$, would give a nontrivial factorization of $p$; a factor that becomes constant would have to be a power of $T$,
whereas $T\nmid p^h$.
Thus $V(p^h)$ has $m_p$ distinct smooth points at infinity, all transverse to $T=0$, and $t=1/z$ is a uniformizer at each of them.

Let $R=U-V$, a polynomial of degree $e$, and let $R^h$ be its degree-$e$ homogenization.
Notice that $R\ne0$, since both $U$ and $V$ have positive degree and depend on different variables.
Equation \eqref{6-8} gives on every infinity branch
\begin{equation}\label{6-15}
R^h\bigl(1,w(t)/z(t),t\bigr) =t^eR(z(t),w(t)) =O(t^{e+1}).
\end{equation}
If $p\nmid R$, equivalently, if $p^h$ and $R^h$ have no common irreducible component, the restriction of $R^h$ to each branch of $V(p^h)$ is nonzero.
At a point $\xi$ at infinity, the local intersection formula \eqref{6-5}, together with \eqref{6-15}, therefore gives
\[
I_\xi(p^h,R^h)\ge e+1.
\]
Hence the $m_p$ points at infinity alone contribute at least
\[
m_p(e+1)>m_pe
\]
to the total intersection number.
This contradicts B\'ezout's theorem \cite[Corollary~I.7.8]{Hartshorne1977}, which gives the total $m_pe$.

We conclude that $p\mid R$.
Since the same argument applies to every irreducible factor of $q$, and since $q$ is square-free, their product divides $R$.
Hence \eqref{6-14} holds.
\end{proof}

\section{Localization and the rank obstruction}\label{sec:localization}

In the coordinate ring $\C[z,w]/(q)$, put $t=U(z)=V(w)$ and $K=\C(t)$.
We shall compute the dimension of the full coordinate ring over $K$ and compare it with the two subspaces coming from $z$ and $w$.
Since $q$ may be reducible, the required freeness must be established before localization.

\begin{proposition}\label{prop7-1}
Let $q\in\C[z,w]$ be a polynomial of total degree $m$ such that
\begin{equation}\label{7-1}
\deg_z q=\deg_w q=m,
\end{equation}
and suppose that the leading coefficient of $q$, as a polynomial in either variable, is a nonzero constant.
Assume that
\begin{equation}\label{7-2}
q\mid U(z)-V(w), \qquad \deg U=\deg V=e\ge1.
\end{equation}
If $m\ge2$, then $q$ does not satisfy $(\AS)$.
\end{proposition}

\begin{proof}
Let
\[
\mathscr A=\C[z,w]/(q),
\]
and let $\pi:\C[z,w]\to\mathscr A$ be the quotient map.
Write
\[
\mathscr A^z=\pi(\C[z]), \qquad \mathscr A^w=\pi(\C[w]).
\]
Inside $\mathscr A$, equation \eqref{7-2} gives a common element
\begin{equation}\label{7-3}
t=U(z)=V(w).
\end{equation}
We regard $\mathscr A$ as a module over $\C[t]$ through this identity.
Since the leading coefficient of $q$ in $w$ is a nonzero constant, we may rescale $q$ so that it is monic of degree $m$ in $w$.
Division in $\C[z][w]$ then gives
\begin{equation}\label{7-4}
\mathscr A=\C[z]\oplus\C[z]w\oplus\cdots\oplus\C[z]w^{m-1}.
\end{equation}
Thus $\mathscr A$ is a free $\C[z]$-module of rank $m$.

We also claim that $\C[z]$ is a free $\C[U(z)]$-module of rank $e$, with basis $1,z,\ldots,z^{e-1}$.
The equation $t=U(z)$ expresses $z^e$ as a $\C[t]$-linear combination of lower powers, and hence reduces every power of $z$ to a $\C[U(z)]$-linear combination of these $e$ elements.
For linear independence, observe that the leading degrees of $a_j(U(z))z^j$, for $0\le j<e$, are distinct modulo $e$; hence the largest one cannot cancel in a nontrivial relation.
Identifying $\C[U(z)]$ with $\C[t]$, we conclude from \eqref{7-4} that
\begin{equation}\label{7-5}
\bigl\{z^iw^j:0\le i<e,\ 0\le j<m\bigr\} \quad\text{is a $\C[t]$-basis of }\mathscr A.
\end{equation}
Thus $\mathscr A$ is free of rank $em$ over $\C[t]$.
In particular, the map $\C[t]\to\mathscr A$ is injective and every nonzero polynomial in $t$ is a nonzero divisor on $\mathscr A$.
This conclusion does not require $q$ to be irreducible or square-free, and localization therefore retains every component.

Let
\[
\Sigma=\C[t]\setminus\{0\}, \qquad K=\Sigma^{-1}\C[t]=\C(t), \qquad \mathscr A_K=\Sigma^{-1}\mathscr A.
\]
Localizing \eqref{7-5} gives the exact dimension
\begin{equation}\label{7-6}
\dim_K\mathscr A_K=em.
\end{equation}

Let
\[
L_z=\Sigma^{-1}\mathscr A^z, \qquad L_w=\Sigma^{-1}\mathscr A^w
\]
inside $\mathscr A_K$.
The natural maps $\C[z]\to\mathscr A$ and $\C[w]\to\mathscr A$ are injective.
Indeed, if a nonzero polynomial $a(z)$ belonged to the ideal $(q)$, then $a(z)=q(z,w)R(z,w)$ for some nonzero $R$, which is impossible after comparing degrees in $w$; the same argument applies to $\C[w]$.
Since $\mathscr A$ is free over $\C[t]$, localization preserves these injections.
Hence
\[
L_z\cong K\otimes_{\C[t]}\C[z], \qquad L_w\cong K\otimes_{\C[t]}\C[w]
\]
inside $\mathscr A_K$
\iffalse
There are natural 
\begin{equation}\label{7-7}
L_z\cong K[Z]/(U(Z)-t), \qquad L_w\cong K[W]/(V(W)-t).
\end{equation}
These are the base changes of $\C[z]$ and $\C[w]$ from $\C[t]$ to $K$.
Moreover, $U(Z)-t$ is irreducible over $K$: a factorization in $\C[Z,t]$ would, by degree in $t$, have a factor independent of $t$, and that factor would divide the coefficient $-1$.
Gauss's lemma now gives the claim over $\C(t)$.
The same argument applies to $V(W)-t$.
Thus $L_z$ and $L_w$ are fields 
\fi 
and
\begin{equation}\label{7-8}
\dim_K L_z=\dim_K L_w=e.
\end{equation}
Their intersection contains the common scalar field $K$, so
\begin{equation}\label{eq:sum-upper}
                 \dim_K(L_z+L_w)
                 =\dim_K L_z+\dim_K L_w-\dim_K(L_z\cap L_w)
                 \le2e-1.
\end{equation}
Assume that $q$ satisfies $(\AS)$.
The sum $\mathscr A^z+\mathscr A^w$ is stable under the multiplication by $t$, since in $\mathscr A$,
\[
t\bigl(A(z)+B(w)\bigr)=U(z)A(z)+V(w)B(w).
\]
We may therefore localize the equality $
\mathscr A=\mathscr A^z+\mathscr A^w$ to obtain
\begin{equation}\label{7-9}
\mathscr A_K=L_z+L_w.
\end{equation}
Hence
$$
\dim_K\mathscr A_K=em\le 2e-1,
$$
contrary to $m\ge2$.
\end{proof}

\begin{remark}\label{rem7-1}
The balanced bidegree hypothesis in Proposition~\ref{prop7-1} is essential.
A polynomial such as $w-z^k$ has $(\AS)$ for every $k$, because its quotient ring is simply $\C[z]$, and it divides the polynomial $w-z^k$.
This cannot happen here, because Lemma~\ref{lem4-1} gives degree $m$ in both variables and only noncoordinate leading lines.
\end{remark}

\begin{proof}[Proof of Theorem~\ref{thm1-2}]
Let $m=\deg q$.
If $m\le2$, there is nothing to prove, so assume $m\ge3$.
By Lemma~\ref{lem4-1}, the slopes $S$ and their ratio group $G(S)=\mu_e$ are defined.

If $e=d$, Lemma~\ref{lem5-2} shows directly that $q$ does not satisfy $(\AS)$, contrary to the hypothesis.

If $e<d$, Lemma~\ref{lem6-2} produces polynomials $U,V$ of degree $e$ such that $q\mid U(z)-V(w)$.
Lemma~\ref{lem4-1}(iii) verifies all balanced bidegree hypotheses of Proposition~\ref{prop7-1}, which again contradicts $(\AS)$.

These two cases exhaust the possibilities for $G(S)$.
Thus $m\ge3$ is impossible, and consequently $\deg q\le2$.
\end{proof}

\section{Proof of Theorem 1.1}\label{sec:completion}

Let $q$ again denote the active polynomial constructed in Section~\ref{sec:active}.
To return to the real plane, we first normalize it so that $q(x+iy,x-iy)$ has real coefficients.
Because $q$ may be reducible, this is done factor by factor.

For $p\in\C[z,w]$, define the conjugate-linear ring involution
\begin{equation}\label{8-1}
p^\#(z,w)=\overline{p(\bar w,\bar z)}.
\end{equation}
On the real slice $w=\bar z$,
\begin{equation}\label{8-2}
p^\#(z,\bar z)=\overline{p(z,\bar z)}.
\end{equation}

\begin{lemma}\label{lem8-1}
Every active irreducible factor can be multiplied by a nonzero scalar so that $p^\#=p$.
Consequently the active product $q$ may be chosen so that
\begin{equation}\label{8-3}
\widetilde q(x,y):=q(x+iy,x-iy) \in\R[x,y].
\end{equation}
\end{lemma}

\begin{proof}
Let $p$ be active.
Whenever $p(z,\bar z)=0$, equation \eqref{8-2} gives $p^\#(z,\bar z)=0$.
Thus $V(p)$ and $V(p^\#)$ have infinitely many common points.
Since $\#$ preserves irreducibility, B\'ezout's theorem implies that they are associates~\cite[Corollary~I.7.8]{Hartshorne1977}:
\[
p^\#=\lambda p
\]
for some $\lambda\in\C^*$.
Applying $\#$ again gives $p=\bar\lambda p^\#=|\lambda|^2p$, and hence $|\lambda|=1$.
Choose $\mu\in\C^*$ such that $\mu/\bar\mu=\lambda$.
Then
\[
(\mu p)^\#=\bar\mu p^\# =\bar\mu\lambda p =\mu p.
\]
Thus replacing $p$ by $\mu p$ makes it fixed by $\#$, and $(\mu p)(x+iy,x-iy)$ has real coefficients.
Applying this normalization to each active factor gives \eqref{8-3}.
The resulting product differs from the original $q$ only by a nonzero scalar, so none of the properties of $q$ used above is changed.
\end{proof}

\begin{proof}[Proof of Theorem~\ref{thm1-1}]
Let \(H,\Psi\), and \(q\) be as in Section~3. By Lemma~3.1 and
Proposition~3.2, \(q\) is a nonconstant square-free divisor of
\[
\Psi(z,w)=zw-H(z)-H^*(w)
\]
and satisfies \((\mathrm{AS})\). If \(\deg\Psi\le 2\), then
\(q\mid\Psi\) gives \(\deg q\le 2\). Otherwise
\(d=\deg H=\deg H^*\ge 3\), and Theorem~1.2, applied with
\(P=-H\), \(Q=-H^*\), and \(c=1\), gives the same conclusion.

By Lemmas~3.1 and~8.1, after rescaling \(q\) we may write
\[
\widetilde q(x,y):=q(x+iy,x-iy)\in\mathbb R[x,y],
\qquad
\partial\Omega\subset V_{\mathbb R}(\widetilde q),
\qquad
\deg\widetilde q\le 2.
\]
Since \(\partial\Omega\) contains a Jordan curve,
\(\widetilde q\) cannot be linear or reducible quadratic.
Thus \(\widetilde q\) is an irreducible quadratic. By the
classification of real affine conics, its real zero set is a
nondegenerate ellipse and denote it by \(J\).

Every component \(\Gamma\) of \(\partial\Omega\) is a Jordan curve
contained in \(J\), and hence \(\partial\Omega=J\). 
Consequently, by the Jordan curve theorem, \(\Omega\) is the bounded interior of the ellipse \(J\).
\end{proof}

\subsection*{Data availability}
No data were used for the research described in this article.

\subsection*{Competing interests}
The authors declare that they have no competing interests.

\subsection*{AI assistance statement}
The authors used OpenAI models for language editing and preliminary proof checks. All mathematical content was independently verified by the authors, who take full responsibility for the manuscript.

\end{document}